\documentclass[a4paper, 10pt]{amsart}
\usepackage{graphicx}
\usepackage{cite}
\usepackage{amsthm,amsmath,amssymb}
\usepackage{mathrsfs}

\newtheorem{theorem}{Theorem}[section]

\newtheorem{lemma}[theorem]{Lemma}
\newtheorem{proposition}[theorem]{Proposition}
\newtheorem{definition}[theorem]{Definition}

\numberwithin{equation}{section}

\begin{document}

\title[]{Vanishing, Unbounded and Angular Shifts on the Quotient of the Difference and the Derivative of a Meromorphic Function}

\thanks{The second author gratefully acknowledges the support of the China Scholarship Council (\#202206820014).}

\author[L. Asikainen]{Lasse Asikainen}
\address[Lasse Asikainen]{Department of Physics and Mathematics, University of Eastern Finland, P.O. Box 111, FI-80101 Joensuu, Finland}
\email{lasseasi@student.uef.fi}

\author[Y. Chen]{Yu Chen$^{*}$}
\address[Yu Chen]{Department of Physics and Mathematics, University of Eastern Finland, P.O. Box 111, FI-80101 Joensuu, Finland}
\email[Corresponding author]{yuchen@student.uef.fi}

\author[R. J. Korhonen]{Risto Korhonen}
\address[Risto Korhonen]{Department of Physics and Mathematics, University of Eastern Finland, P.O. Box 111, FI-80101 Joensuu, Finland}
\email{risto.korhonen@uef.fi}

\date{}



\begin{abstract}
We show that for a vanishing period difference operator of a meromorphic function \( f \), there exist the following estimates regarding proximity functions,
\[
\lim_{\eta \to 0} m_\eta\left(r, \frac{\Delta_\eta f - a\eta}{f' - a} \right) = 0
\]
and
\[
\lim_{r \to \infty} m_\eta\left(r, \frac{\Delta_\eta f - a\eta}{f' - a} \right) = 0,
\]
where \( \Delta_\eta f = f(z + \eta) - f(z) \), and \( |\eta| \) is less than an arbitrarily small quantity \( \alpha(r) \) in the second limit. Then, under certain assumptions on the growth, restrictions on the period tending to infinity, and on the value distribution of a meromorphic function \( f(z) \), we have
\[
m\left(r, \frac{\Delta_\omega f - a\omega}{f' - a} \right) = S(r, f'),
\]
as \( r \to \infty \), outside an exceptional set of finite logarithmic measure.

Additionally, we provide an estimate for the angular shift under certain conditions on the shift and the growth. That is, the following Nevanlinna proximity function satisfies
\[
m\left(r, \frac{f(e^{i\omega(r)}z) - f(z)}{f'} \right) = S(r, f),
\]
outside an exceptional set of finite logarithmic measure.

Furthermore, the above estimates yield additional applications, including deficiency relations between \( \Delta_\eta f \) (or \( \Delta_\omega f \)) and \( f' \), as well as connections between \( \eta/\omega \)-separated pair indices and \( \delta(0, f') \).
\end{abstract}
\maketitle

\section{Introduction and main results}

The relationship between the value distribution properties of the derivative $f^\prime$ and those of the finite difference $\Delta_c f(z)=f(z+c)-f(z)$, where $c\in\mathbb{C}\setminus\{0\}$, is a subject of considerable interest. This connection is underscored by the fact that $\frac{\Delta_cf}{c}\to f^\prime$ pointwise as $c\to0$. In 2007, Bergweiler and Langley \cite{MR2296397} investigated the behaviour of the finite difference operator, demonstrating that for a fixed $c\in\mathbb{C}\setminus\{0\}$, $\frac{\Delta_cf}{c}\to f^\prime$ as $|z|\to\infty$, except on a set of points with density approaching zero, if $f$ is a meromorphic function of order less than $1$. This result highlights a deep connection between the two operators.

Subsequent studies have focused on uncovering weaker but more general links between $f^\prime$ and $\Delta_cf$, applicable to broader classes of meromorphic functions. In this context, significant attention has been given to the difference quotient $f(z + c)/f(z)$. For example, Chiang and Feng \cite{MR2491899} established an asymptotic relationship between the difference quotient and the logarithmic derivative for meromorphic functions of finite order, providing further insight into the interplay between these operators.

In 2023, Asikainen, Huusko, and Korhonen \cite{asikainen2023new} established a result relating the value distributions of the derivative \( f' \) and the difference \( \Delta_c f = f(z + c) - f(z) \) of a non-\( c \)-periodic meromorphic function \( f \) with hyper-order \( \xi < \frac{3}{4} \). In particular, they showed that the following proximity function satisfies
\begin{equation}
\label{eq:asikainenetal_mainthm}
m\left(r, \frac{\Delta_c f - ac}{f' - a} \right) = S(r, f),
\end{equation}
where \( a \in \mathbb{C} \) is arbitrary, and \( S(r, f) \) denotes a quantity small relative to \( f \), with the estimate holding outside an exceptional set of finite logarithmic measure.

The above result only concerns a fixed shift $c\in\mathbb{C}$. In light of the results by Chiang and Ruijsenaars \cite{MR2220338}, who demonstrated that for a non-zero meromorphic function $f(z)$ and $c\in\mathbb{C}$, the following inequality holds:
$$
m(r,f(z+c))<\frac{R+2r}{R-2r}m(R,f)+\sum_{l=0}^L\frac{1}{2\pi}\int_0^{2\pi}\log\left|\frac{R^2-\bar{b}_l(re^{i\phi}+c)}{R(re^{i\phi}+c-b_l)}\right|{\rm d}\phi,
$$
where $|c|<r$, and $b_0,\dots,b_L$ are poles of $f (z)$ in $|z|<R$, a uniform bound can be derived:
$$
m(r,f(z+c))\le5m(3r,f)+n(4r,f)\cdot\log4
$$
whenever $|c|<r$. This raises the natural question of whether the results obtained by Asikainen et al. can be extended to the case of variable shift.

In 2017, Chiang and Luo \cite{MR3695346} advanced the field of difference Nevanlinna theory by developing frameworks for two distinct scenarios: meromorphic functions with steps approaching zero (vanishing period) and finite-order meromorphic functions with steps tending to infinity (infinite period). Their work extended the traditional concept of a difference operator with a fixed step to accommodate operators with varying steps, enabling the analysis of a broader class of functions and behaviours.

We generalize the fixed-step difference operator by introducing a varying-step difference operator, building upon the framework established by Asikainen, Huusko, and Korhonen \cite{asikainen2023new}. We now turn our attention to the vanishing and unbounded shifts, which offer an alternative extension of Theorem 2.2 from \cite{asikainen2023new} and equation~(\ref{eq:asikainenetal_mainthm}). We begin by considering the case of a vanishing shift.

\begin{theorem}
    \label{t1.1}
    Let $f(z)$ be a non-constant meromorphic function in $\mathbb{C}$, $a\in\mathbb{C}$ and $r=|z|$ be fixed. Then we have
    $$
    \lim_{\eta\to0}m_\eta\left(r,\frac{\Delta_\eta f-a\eta}{f'-a}\right)=0.
    $$
    Moreover, if $|\eta|<\alpha(r)$, where $\alpha(r)$ is an arbitrary term such that $\alpha(r)\to0$ as $r\to\infty$, then
    $$
    \lim_{r\to\infty}m_\eta\left(r,\frac{\Delta_\eta f-a\eta}{f'-a}\right)=0.
    $$
\end{theorem}

Secondly, we consider applying an unbounded shift with low growth, rather than a constant (ordinary) shift. This leads us to the following theorem.

\begin{theorem}
    \label{t1.2}
    Let $f(z)$ be a non-constant meromorphic function of hyper-order $\varsigma<\frac{3}{4}$, $0<\beta<\min\left\{\frac{1}{2}-\frac{1}{2}\varsigma,1-\frac{4}{3}\varsigma\right\}$, $0<|\omega(r)|<r^\beta$, and $a\in\mathbb{C}$. Then, for $\epsilon>0$ being sufficiently small, we have
    $$
    \begin{aligned}
        m\left(r,\frac{\Delta_{\omega }f-a\omega }{f'-a}\right)=&O\left(\frac{T(r,f')}{r^{1-\varsigma-2\beta-\epsilon}}\right)
        +O\Bigg(R_{\epsilon,\omega }\left(r,\frac{1}{f'-a}\right)\frac{N\left(r,\frac{1}{f'-a}\right)}{r^{\frac{3}{2}-2\varsigma-\frac{3}{2}\beta-\epsilon}}\\
        &+R_{\epsilon,\omega }(r,f')\frac{N(r,f')}{r^{\frac{3}{2}-2\varsigma-\frac{3}{2}\beta-\epsilon}}\Bigg)+O(\log r)\\
        =&S(r,f^\prime)
    \end{aligned}
    $$
    as $r\to\infty$ outside an exceptional set $E=E(\epsilon,a,\omega,f)$ of finite logarithmic measure, where
    $$
    R_{\epsilon,\omega }(r,g)=\frac{n_{\angle\epsilon,\omega }(r,g)}{n(r,g)}
    $$
    with $n_{\angle\epsilon,\omega }(r,g)$ counting a pole $\left|z_0\right|<r$ of $g$ according to its multiplicity only if
    $$
    \left|\sin\left({\rm arg}\frac{z_{0}}{\omega }\right)\right|\geq1-\sqrt{\epsilon}.
    $$
\end{theorem}
Similar results hold when an angular shift is considered in place of an ordinary shift, leading to the following theorem.

\begin{theorem}
    \label{t.angular}
    Let $f(z)$ be a meromorphic function. Suppose the angular shift $\omega(r)>0$ satisfies 
    $$\omega(r)\leq\frac{1}{T(r+\varepsilon,f)^\frac{3\varepsilon}{1+2\varepsilon}}\cdot\frac{1}{r^{\frac{\varepsilon}{1+\varepsilon}%+(2+\varepsilon)\varsigma
    }}$$
    for any $\varepsilon>0$.
    Then the following estimate
    \begin{align*}
    \quad m\left(r,\frac{f(e^{i\omega(r)}z)-f(z)}{f'}\right)=S(r,f)
    \end{align*}
    holds for all $r$ outside an exceptional set with finite logarithmic measure.
\end{theorem}

The remainder of this paper is structured as follows. Section 2 provides a review of the fundamental notation and key results from Nevanlinna theory, along with the presentation and detailed proof of a central lemma. Sections 3 through 5 contain the proofs of Theorems \ref{t1.1}, \ref{t1.2}, and \ref{t.angular}, respectively. Finally, Section 6 presents several consequences that are closely related to Theorems \ref{t1.1}, \ref{t1.2}, and \ref{t.angular}.

%Section 2 provides a review of fundamental notations and key results in Nevanlinna theory. Meanwhile, a pivotal lemma is presented and rigorously proved. 
%\begin{itemize}
%    \item {Section 2} provides a review of fundamental notations and key results in Nevanlinna theory. Meanwhile, a pivotal lemma is presented and rigorously proved.
%    \item {Section 3-5} provide the proofs of Theorems \ref{t1.1}, \ref{t1.2} and \ref{t.angular}, respectively.
%    \item {Section 6} presents several consequences that are closely related to Theorems \ref{t1.1}, \ref{t1.2} and \ref{t.angular}.
%\end{itemize}
    
\section{Preliminaries}

We now briefly introduce some basic notation and results from Nevanlinna theory. For further details, the reader may consult references such as \cite{MR1831783}, \cite{MR220938}, and \cite{MR0164038}.\par

In Nevanlinna theory, we are often concerned with the asymptotic behaviour of the characteristic function of a given meromorphic function, described by concepts such as order, hyper-order, and subnormal growth. Many results in Nevanlinna theory hold for most values of \( r \), with the exceptions forming what is known as the exceptional set. In this paper, we primarily focus on exceptional sets \( E \) of finite logarithmic measure, defined as those for which $\int_E\frac{1}{t}\mathrm{d}t<\infty$.\par

Given a function \( f \) meromorphic in the entire complex plane, we introduce the following notation. A meromorphic function \( g \) is said to be a \emph{small function} with respect to another meromorphic function \( f \) if
\[
T(r, g) = o(T(r, f))
\quad \text{as } r \to \infty,
\]
outside an exceptional set of finite logarithmic measure. Any such small error term is denoted by \( S(r, f) \); in this case, we write \( T(r, g) = S(r, f) \). Intuitively, this means that the characteristic function of \( g \) is much smaller than that of \( f \) for most values of \( r \).

The \emph{order} of $f$ is defined by
    \begin{equation*}
    \rho(f):=\limsup_{r\to \infty }\frac{\log^+ T(r,f)}{\log r}.
\end{equation*}

The following lemma, established by Halburd, Korhonen, and Tohge in \cite{MR3206459}, provides a crucial estimate for the differences of meromorphic functions with hyper-order less than 1. As a result, it is utilized multiple times throughout this paper.

\begin{lemma}[Lemma 8.3 from \cite{MR3206459}]
    \label{l2.3}
    Let $T:[0,+\infty)\to[0,+\infty)$ be an increasing continuous function of hyper-order strictly less than one, i.e. 
    $$
    \limsup_{r\to\infty}\frac{\log\log T(r)}{\log r}=\varsigma<1.
    $$
    Then if $u>0$ is fixed, we have
    $$
    T(r+u)-T(r)=o\left(\frac{T(r)}{r^\tau}\right)
    $$
    where $\tau\in(0,1-\varsigma)$, and $r$ runs to infinity outside an exceptional set of finite logarithmic measure.
\end{lemma}

The following lemma generalizes Lemma \ref{l2.3} and \cite[Lemma 3.2]{MR2526825}, and will be applied in the proof of Theorem \ref{t1.2}.

\begin{lemma}
\label{l2.4}
    Let $T:[0,+\infty)\to[0,+\infty)$ be an increasing continuous function that is of hyper-order $\varsigma<1$, and let $0<|\omega(r)|<r^\beta$ $(0<\beta<1-\varsigma)$. Then
    $$
    T(r+\omega)-T(r)=o\left(\frac{T(r)}{r^{\tilde\tau}}\right)
    $$
    where $\tilde\tau\in(0,1-\varsigma-\beta)$, and $r$ runs to infinity outside an exceptional set of finite logarithmic measure.
\end{lemma}

\begin{proof}
    Let $\tilde\tau\in(0,1-\varsigma-\beta)$, $\gamma\in\mathbb{R}^+$ and assume that the set $$F_\gamma=\left\{r\in\mathbb{R}^+:\frac{T(r+\omega)-T(r)}{T(r)}\cdot r^{\tilde\tau}\geq\gamma\right\}$$
    is of infinite logarithmic measure. Note that $F_\gamma$ is a closed set and therefore it has a smallest element, say $r_0$. Set $r_{n}= \operatorname*{min}\{F_\gamma\cap[r_{n-1}+\omega(r),\infty)\}$ for all $n\in\mathbb{N}$. Then, the sequence $\{r_n\}_{n\in\mathbb{N}}$ satisfies $r_{n+1}-r_n\geq \omega(r)$ for all $n\in\mathbb{N}$, $F_\gamma\subset\bigcup_{n=0}^\infty[r_n,r_n+\omega(r)]$ and
    \begin{equation}
        \label{eq2.1}
        \left(1+\frac\gamma{r_n^{\tilde{\tau}}}\right)T(r_n)\leq T(r_{n+1}),
    \end{equation}
    for all $n\in\mathbb{N}$.

    Let $\varepsilon>0$, and suppose that there exist an $m\in\mathbb{N}$ such that $r_n\ge n^\frac{1+\varepsilon}{1-\beta}$ and $|\omega(r)|<r_n^\beta$ for all $r_n\ge m$. But then,
    $$
    \begin{aligned}
\int_{F}\frac{\mathrm{d}t}{t} &\leq\sum_{n=0}^\infty\int_{r_n}^{r_n+|w|}\frac{\mathrm{d}t}t\leq\int_1^m\frac{\mathrm{d}t}t+\sum_{n=1}^\infty\log\left(1+\frac{\omega(r)}{r_n}\right) \\
&\leq\sum_{n=1}^\infty\log(1+n^{-(1+\varepsilon)})+O(1)<\infty, 
\end{aligned}
    $$
    which contradicts the assumption $\int_F(\mathrm{d}t/t)=\infty $. Therefore, the sequence $\{r_n\}_{n\in\mathbb{N}}$ has a subsequence $\{r_{n_j}\}_{j\in\mathbb{N}}$ such that $r_{n_{j}}\leq n_{j}^\frac{1+\varepsilon}{1-\beta}$ for all $j\in\mathbb{N}$. But iterating \eqref{eq2.1} along the sequence $\{r_n\}_{n\in\mathbb{N}}$, it follows that
    $$
    T(r_{n_j})\geq\prod_{\nu=0}^{n_j-1}\left(1+\frac{\gamma}{{r_\nu}^{\tilde{\tau}}}\right)T(r_0),
    $$
    for all $j\in\mathbb{N}$, and hence
    $$
    \begin{aligned}
\limsup_{r\to\infty}\frac{\log\log T(r)}{\log r}& \geq\limsup_{j\to\infty}\frac{\log\log T(r_{n_j})}{\log r_{n_j}} \\
&\geq\limsup_{j\to\infty}\frac{\log\left(\log T(r_0)+\sum_{\nu=0}^{n_j-1}\log\left(1+\frac{\gamma}{{r_\nu}^{\tilde{\tau}}}\right)\right)}{\log r_{n_j}} \\
&\geq\limsup_{j\to\infty}\frac{\log\left(\log T(r_0)+n_j\log\left(1+\frac\gamma{r_{n_j}^{\tilde{\tau}}}\right)\right)}{\frac{1+\varepsilon}{1-\beta}\log n_j}\\
&\ge\limsup_{j\to\infty}\frac{\log\left(\log T(r_0)+n_j\frac{\gamma}{n_j^{\frac{1+\varepsilon}{1-\beta}\tilde{\tau}}}\log\left(1+\frac{\gamma}{n_j^{\frac{1+\varepsilon}{1-\beta}\tilde{\tau}}}\right)^{\frac{n_j^{\frac{1+\varepsilon}{1-\beta}\tilde{\tau}}}{\gamma}}\right)}{\frac{1+\varepsilon}{1-\beta}\log n_j}\\
&\ge\limsup_{j\to\infty}\frac{\left(1-\frac{1+\varepsilon}{1-\beta}\tilde{\tau}\right)\log n_j}{\frac{1+\varepsilon}{1-\beta}\log n_j}\\
&\ge\frac{1-\beta}{1+\varepsilon}-\tilde{\tau}.
\end{aligned}
    $$
    By letting $\varepsilon\to0$, we obtain
    $$
    \limsup_{r\to\infty}\frac{\log\log T(r)}{\log r}\geq1-\beta-\tilde{\tau},
    $$ 
    This contradicts our assumption on $\tilde{\tau}$, so the assertion follows. 
\end{proof}

\section{The proof of Theorem \ref{t1.1}}

Let $S_1\subset [0,2\pi]$ be the subset such that for all $\theta\in [0,2\pi]\setminus S_1$, the line segment $[re^{i\theta},re^{i\theta}+\eta]$ contains poles or $a$-points of $f'$. Since the poles and zeros of a meromorphic function are isolated, they must be at most countable. We can construct non-intersecting open disks $B(\hat{z})$ with each zero or pole $\hat{z}$ as the centre, make a mapping $\hat{z}\to R(\hat{z})$, where $R(\hat{z})$ is the set of all rational points in $B(\hat{z})$. Then it is obvious that the cardinality of $\{\hat{z}\}$ does not exceed the cardinality of all rational points, which is at most countable. Thus, $[0,2\pi]\setminus S_1$ is at most countable and is a zero-measure set. Therefore, 
\begin{equation}
    \label{eq4.1}
    \begin{aligned}\int_0^{2\pi}\log^+\left|\frac{\Delta_\eta f(re^{i\theta})-a\eta}{f'(re^{i\theta})-a}\right|\frac{\mathrm{d}\theta}{2\pi}&=\int_{S_1}\log^+\left|\frac{\Delta_\eta f(re^{i\theta})-a\eta}{f'(re^{i\theta})-a}\right|\frac{\mathrm{d}\theta}{2\pi}.\end{aligned}
\end{equation}
Now, we have the following inequality by Lagrange's Mean Value Theorem,
\begin{equation}
    \label{eq4.2}
    \begin{aligned}\left|\frac{\Delta_\eta f(re^{i\theta})-a\eta}{f'(re^{i\theta})-a}\right|=\left|\int_0^\eta\frac{f'(re^{i\theta}+u)-a}{f'(re^{i\theta})-a}\mathrm{d}u\right|\leq|\eta|\max_{t\in[0,1]}\left|\frac{g'(re^{i\theta}+t\eta)}{g'(re^{i\theta})}\right|,\end{aligned}
\end{equation}
where $g(z)=f(z)-az$. We see that $\left|\frac{\Delta_\eta f(re^{i\theta})-a\eta}{f'(re^{i\theta})-a}\right|<1$ when $\eta\to0$. Therefore, 
$$
\lim_{\eta\to0}m_\eta\left(r,\frac{\Delta_\eta f-a\eta}{f'-a}\right)=0.
$$

On the other hand, if we set $|\eta|<\alpha(r)$, where $\alpha(r)$ is an arbitrary function such that $\alpha(r)\to0$ as $r\to\infty$, we obtain 
$$
\lim_{r\to\infty}m_\eta\left(r,\frac{\Delta_\eta f-a\eta}{f'-a}\right)=0.
$$

\section{The proof of Theorem \ref{t1.2}}

Let $S_2\subset [0,2\pi]$ be the subset such that for all $\theta\in [0,2\pi]\setminus S_2$, the line segment $[re^{i\theta},re^{i\theta}+\omega]$ contains poles or $a$-points of $f'$. We can apply similar reasoning as in the previous section to conclude that $[0,2\pi]\setminus S_2$ is at most countable and has measure zero. Therefore, 
\begin{equation}
    \label{eq5.1}
    \begin{aligned}\int_0^{2\pi}\log^+\left|\frac{\Delta_\omega f(re^{i\theta})-a\omega}{f'(re^{i\theta})-a}\right|\frac{\mathrm{d}\theta}{2\pi}&=\int_{S_2}\log^+\left|\frac{\Delta_\omega f(re^{i\theta})-a\omega}{f'(re^{i\theta})-a}\right|\frac{\mathrm{d}\theta}{2\pi}.\end{aligned}
\end{equation}
Next we apply the Poisson-Jensen formula for which we have to recall the Poisson-kernel 
$$P(z,\theta)=
\frac{1-|z|^2}{|e^{i\theta}-z|^2}=\operatorname{Re}\left\{\frac{e^{i\theta}+z}{e^{i\theta}-z}\right\}$$
and the Green function
$$G(z,a)=\log\left|\frac{1-\overline{a}z}{z-a}\right|.$$

Using the estimate \eqref{eq4.2} in \eqref{eq5.1} and then applying the Poisson-Jensen formula, we obtain
\begin{equation}
    \label{eq5.2}
    \begin{gathered}
\int_{0}^{2\pi}\log ^{+}\left|\frac{\Delta_{\omega}f(re^{i\theta})-a\omega}{f'(re^{i\theta})-a}\right|\frac{\mathrm{d}\theta}{2\pi}\leq\int_{S_2}\operatorname*{max}_{t\in[0,1]}\log \left|\frac{g'(re^{i\theta}+t\omega)}{g'(re^{i\theta})}\right|\frac{\mathrm{d}\theta}{2\pi}+\log \omega \\
 \begin{aligned}&\leq\int_0^{2\pi}\max_{t\in[0,1]}\Bigg\{\int_0^{2\pi}\log|g'(se^{i\psi})|\left(P\left(\frac{re^{i\theta}+t\omega}{s},\psi\right)-P\left(\frac{re^{i\theta}}{s},\psi\right)\right)\frac{\mathrm{d}\psi}{2\pi} \\
&\quad\quad+\sum_{|a_k|<s}G\left(\frac{re^{i\theta}}{s},\frac{a_k}{s}\right)-G\left(\frac{re^{i\theta}+t\omega}{s},\frac{a_k}{s}\right)\\
&\quad\quad+\sum_{|b_k|<s}G\left(\frac{re^{i\theta}+t\omega}{s},\frac{b_k}{s}\right)-G\left(\frac{re^{i\theta}}{s},\frac{b_k}{s}\right)\bigg\}+\beta\log r
\end{aligned}\\
\quad\begin{aligned}&\leq\int_0^{2\pi}\Bigg\{\max_{t\in[0,1]}\int_0^{2\pi}\log|g'(se^{i\psi})|\text{Re}\left(\frac{2t\omega se^{i\psi}}{(se^{i\psi}-re^{i\theta}-t\omega)(se^{i\psi}-re^{i\theta})}\right)\frac{\mathrm{d}\psi}{2\pi} \\
&\quad+\max_{t\in[0,1]}\sum_{|a_k|<s}\log\left|\frac{re^{i\theta}+t\omega-a_k}{re^{i\theta}-a_k}\right|+\max_{t\in[0,1]}\sum_{|a_k|<s}\log\left|\frac{s^2-\overline{a}_k re^{i\theta}}{s^2-\overline{a}_k (re^{i\theta}+t\omega )}\right| \\
&\quad+\max_{t\in[0,1]}\sum_{|b_k|<s}\log\left|\frac{re^{i\theta}-b_k}{re^{i\theta}+t\omega -b_k}\right|+\max_{t\in[0,1]}\sum_{|b_k|<s}\log\left|\frac{s^2-\overline{b_k}(re^{i\theta}+t\omega)}{s^2-\overline{b_k}re^{i\theta}}\right|\bigg\}\frac{\mathrm{d}\theta}{2\pi}\\
&\quad+\beta\log r
\end{aligned}
\end{gathered}
\end{equation}
where $s=\frac{\alpha+1}{2}(r+|\omega|)$ with $\alpha:=\alpha(r)>1$ and $\alpha(r)\to1$ as $r\to\infty$, and where $\{a_k\}_{k\in\mathbb{N}}$ and $\{b_k\}_{k\in\mathbb{N}}$ are sequences of the zeros and poles of $g'$ respectively, ordered by modulus in ascending order and repeated according to their multiplicities.

We use the well-known fact that 
\begin{equation}
    \label{eq5.3}
    \begin{aligned}
        \int_0^{2\pi}\frac1{|re^{i\theta}-a|^\delta}\leq\frac1{1-\delta}\frac1{r^\delta}
    \end{aligned}
\end{equation}
for any $a\in\mathbb{C}$ and any $\delta\in(0,1)$, in order to estimate term 1 on the RHS (right-hand side) of \eqref{eq5.2}
\begin{equation}
    \label{eq5.4}
\begin{aligned}&\quad \int_0^{2\pi}\max_{t\in[0,1]}\int_0^{2\pi}\log|g'(se^{i\psi})|\text{Re}\left(\frac{2t\omega se^{i\psi}}{(se^{i\psi}-re^{i\theta}-t\omega )(se^{i\psi}-re^{i\theta})}\right)\frac{\mathrm{d}\psi}{2\pi}\frac{\mathrm{d}\theta}{2\pi}\\&\leq\int_0^{2\pi}\int_0^{2\pi}\left|\log|g'(se^{i\psi})|\right|\max_{t\in[0,1]}\left\{\frac{2t|\omega |s}{|se^{i\psi}-re^{i\theta}-t\omega ||se^{i\psi}-re^{i\theta}|}\right\}\frac{\mathrm{d}\psi}{2\pi}\frac{\mathrm{d}\theta}{2\pi}\\
&\le\frac{2|\omega |s}{(s-r-|\omega |)(s-r)^{1-\delta}}\int_{0}^{2\pi}\left|\log |g'(se^{i\psi})|\right|\int_{0}^{2\pi}\frac{1}{|re^{i\theta}-se^{i\psi}|^{\delta}}\frac{\mathrm{d}\theta}{2\pi}\frac{\mathrm{d}\psi}{2\pi} \\
&\le\frac{2|\omega |^{\delta}s}{(s-r-|\omega |)(1-\delta)r^{\delta}}(m(s,g')+m(s,1/g')) \\
&\le\frac{4r^{\delta(\beta-1)}(\alpha+1)}{(1-\delta)(\alpha-1)}(T(\alpha(r+|\omega |),g')+O(1)).
\end{aligned}
\end{equation}

For term 2 on the RHS of \eqref{eq5.2}, we may swap the order of the integral and the sum, since the sum is ﬁnite. Again we employ the estimate \eqref{eq5.3}, as well as using the estimate $\log(1 + x) \le x$ and the concavity of the logarithm:
\begin{equation}
    \label{eq5.5}
    \begin{aligned}
&\quad\int_0^{2\pi}\max_{t\in[0,1]}\sum_{|a_k|<s}\log\left|\frac{re^{i\theta}+t\omega -a_k}{re^{i\theta}-a_k}\right|\frac{\mathrm{d}\theta}{2\pi} \\
&\leq\sum_{|a_k|<s}\frac1\delta\int_0^{2\pi}\max_{t\in[0,1]}\log\left(1+\left|\frac{t\omega }{re^{i\theta}-a_k}\right|^\delta\right)\frac{\mathrm{d}\theta}{2\pi} \\
%&\leq\sum_{|a_k|<s}\frac1\delta\log\left(\int_0^{2\pi}1+\max_{t\in[0,1]}\left|\frac{t\omega }{re^{i\theta}-a_k}\right|^\delta\frac{\mathrm{d}\theta}{2\pi}\right) \\
&\leq\sum_{|a_k|<s}\frac1\delta\log\left(1+|\omega |^\delta\int_0^{2\pi}\frac1{|re^{i\theta}-a_k|^\delta}\frac{\mathrm{d}\theta}{2\pi}\right) \\
&\leq\frac{|\omega |^\delta}{\delta(1-\delta)}\frac{1}{r^\delta}n(s,1/g')\\%\leq\frac{2|\omega |^\delta}{\delta(1-\delta)}\frac{\alpha}{\alpha-1}\frac{1}{r^\delta}N(\alpha(r+|\omega |),1/g')\\
&\leq\frac{|\omega |^\delta}{\delta(1-\delta)}\frac{1}{r^\delta}\frac{2\alpha}{\alpha-1}\int_{s}^{\alpha r}\frac{n(s,1/g')}{t}dt\\
&\leq\frac{2}{\delta(1-\delta)}\frac{\alpha}{\alpha-1}r^{\delta(\beta-1)} N(\alpha(r+|\omega |),1/g').
\end{aligned}
\end{equation}

For term 5 we use the same arguments to obtain a similar estimate:
\begin{equation}
    \label{eq5.6}
    \begin{aligned}
&\quad\int_0^{2\pi}\max_{t\in[0,1]}\sum_{|b_k|<s}\log\left|\frac{s^2-\overline{b_k}(re^{i\theta}+t\omega )}{s^2-\overline{b_k}re^{i\theta}}\right|\frac{\mathrm{d}\theta}{2\pi} \\
&\leq\sum_{|b_k|<s}\int_0^{2\pi}\log\left(1+|\omega |\frac{1}{|re^{i\theta}-s^2/\overline{b_k}|}\right)\frac{\mathrm{d}\theta}{2\pi}\\
%&\le\frac{2|\omega |^\delta}{\delta(1-\delta)}\frac{\alpha}{\alpha-1}\frac{1}{r^\delta}N(\alpha(r+|\omega |),g')\\
&\leq\frac{2}{\delta(1-\delta)}\frac{\alpha}{\alpha-1}r^{\delta(\beta-1)} N(\alpha(r+|\omega |),g').
\end{aligned}
\end{equation}

Let us fix $\epsilon\in(0,1)$.  For the 3rd term, we divide the sum over the poles of $g'$ into two cases:
$$
\begin{aligned}
    &(1)\ Z_\epsilon^1=\{a_k: k\in\mathbb{N}\text{ such that }|a_k|\geq s-\epsilon\};\\
    &(2)\ Z_\epsilon^2=\{a_k: k\in\mathbb{N}\text{ such that }a_k\not\in Z_\epsilon^1\}.
\end{aligned}
$$
Similarly, we divide the sum in the 4th term into two cases:
$$
\begin{aligned}
    \begin{array}{l}(1)\ P_\epsilon^1=\{b_k: k\in\mathbb{N}\text{ such that }r-|\omega|-\epsilon\leq|b_k|\leq r+|\omega|+\epsilon\};\\
    (2)\ P_\epsilon^2=\{b_k: k\in\mathbb{N}\text{ such that }b_k\not\in P_\epsilon^1\}.\end{array}
\end{aligned}
$$

Let us estimate the contribution to term 4 of the poles in $P_\epsilon^2$. We have the following estimate for the maximum of the sum
$$
\begin{aligned}
&\max_{t\in[0,1]}\sum_{\substack{|b_{k}|<s\\b_{k}\in P_{\epsilon}^{2}}}\log\left|\frac{re^{i\theta}-b_k}{re^{i\theta}+t\omega -b_k}\right|\leq\frac{1}{\delta}\sum_{\substack{|b_{k}|<s\\b_{k}\in P_{\epsilon}^{2}}}\log\left(1+\max_{t\in[0,1]}\left|\frac{t\omega }{re^{i\theta}+t\omega -b_k}\right|^\delta\right) \\
\leq&\frac{|\omega |^{\delta}}{\delta} \sum_{\substack{|b_{k}|<s\\b_{k}\in P_{\epsilon}^{2}}}\operatorname*{max}_{t\in[0,1]}\left|\frac{1}{re^{i\theta}+t\omega -b_{k}}\right|^{\delta}\leq\frac{|\omega |^{\delta}}{\delta} \sum_{\substack{|b_{k}|<s\\b_{k}\in P_{\epsilon}^{2}}}\operatorname*{max}_{t\in\overline{\mathbb{D}}}\left|\frac{1}{re^{i\theta}+t\omega -b_{k}}\right|^{\delta} \\
\leq&\frac{|\omega |^\delta}{\delta}\sum\limits_{\substack{|b_k|<s\\b_k\in P_\epsilon^2}}\left|\frac{1}{re^{i\theta}+|\omega |\frac{b_k-re^{i\theta}}{|b_k-re^{i\theta}|}-b_k}\right|^\delta,
\end{aligned}
$$
since $re^{i\theta}+t\omega$ lies inside the closed annulus $r-|\omega|\le|z|\le r+|\omega|$ for every $t\in\overline{\mathbb{D}}$ for all large enough $r$ and $b_k\in P_\epsilon^2$ lies outside that annulus so that the optimal choice of $t$ is the unit vector in the direction of $b_k$ from $re^{i\theta}$. Further, since $b_k$ lies outside the annulus, we have the estimate
$$
\begin{aligned}\left|re^{i\theta}+|\omega |\frac{b_k-re^{i\theta}}{|b_k-re^{i\theta}|}-b_k\right|&=\left|1-\frac{|\omega |}{|b_k-re^{i\theta}|}\right|\left|re^{i\theta}-b_k\right|\\&\geq\left(1-\frac{|\omega |}{|\omega |+\epsilon}\right)\left|re^{i\theta}-b_k\right| ,\end{aligned}
$$
so that in total we obtain the estimate
\begin{equation}
    \label{eq5.7}
    \begin{aligned}
&\int_0^{2\pi}\max_{t\in[0,1]}\sum\limits_{\substack{|b_k|<s\\b_k\in P_\epsilon^2}}\log\left|\frac{re^{i\theta}-b_k}{re^{i\theta}+t\omega -b_k}\right|\frac{\mathrm{d}\theta}{2\pi} \\
\le&\frac{|\omega |^\delta}{\delta}\left(\frac{|\omega |+\epsilon}{\epsilon}\right)^\delta\sum\limits_{\substack{|b_k|<s\\b_k\in P_\epsilon^2}}\int_0^{2\pi}\frac{1}{|re^{i\theta}-b_k|^\delta}\frac{\mathrm{d}\theta}{2\pi} \\
\leq&\left(\frac{|\omega |+\epsilon}{\epsilon}\right)^\delta\frac{|\omega |^\delta}{\delta(1-\delta)}\frac{1}{r^\delta}n(s,g') \\
%\leq&\left(\frac{|\omega |+\epsilon}{\epsilon}\right)^\delta\frac{2|\omega |^\delta}{\delta(1-\delta)}\frac{\alpha}{\alpha-1}\frac{1}{r^\delta}N(\alpha(r+|\omega |),g')\\
\le&\left(\frac{1}{\epsilon}\right)^\delta\frac{2}{\delta(1-\delta)}\frac{\alpha}{\alpha-1}r^{\delta(2\beta-1)}N(\alpha(r+|\omega |),g').
\end{aligned}
\end{equation}

Using the same reasoning to estimate the contribution to term 4 of the zeros in $P_\epsilon^2$, we obtain the following upper bound
\begin{equation}
    \label{eq5.8}
    \begin{aligned}
&\int_0^{2\pi}\max_{t\in[0,1]}\sum_{\substack{|a_{k}|<s\\a_{k}\in Z_{\epsilon}^{2}}}\log\left|\frac{s^2-\overline{a}_k re^{i\theta}}{s^2-\overline{a}_k (re^{i\theta}+t\omega )}\right|\frac{\mathrm{d}\theta}{2\pi} \\
\le&\frac{|\omega |^{\delta}}{\delta}\left(\frac{|\omega |+\epsilon}{\epsilon}\right)^{\delta}\sum_{\substack{|a_{k}|<s\\a_{k}\in Z_{\epsilon}^{2}}}\int_{0}^{2\pi}\frac{1}{|re^{i\theta}-s^{2}/\overline{{a_{k}}}|^{\delta}}\frac{\mathrm{d}\theta}{2\pi} \\
\leq&\left(\frac{|\omega |+\epsilon}{\epsilon}\right)^\delta\frac{2|\omega |^\delta}{\delta(1-\delta)}\frac{\alpha}{\alpha-1}\frac{1}{r^\delta}N(\alpha(r+|\omega |),1/g')\\
\le&\left(\frac{1}{\epsilon}\right)^\delta\frac{2}{\delta(1-\delta)}\frac{\alpha}{\alpha-1}r^{\delta(2\beta-1)}N(\alpha(r+|\omega |),1/g'),
\end{aligned}
\end{equation}
because our assumption $|a_k|<s-\epsilon$ for points $a_{k}\in Z_{\epsilon}^{2}$ implies that $s^2/\overline{a}_k >r+|\omega|+\epsilon$, thus allowing us to use the same reasoning that we used for term 4.

Using \cite[Lemma 3.1]{asikainen2023new} we obtain the following estimates for the $Z_\epsilon^1$ and $P_\epsilon^1$ parts of terms 3 and 4,
\begin{equation}
    \label{eq5.9}
    \begin{aligned}
&\int_0^{2\pi}\max_{t\in[0,1]}\sum_{\substack{|a_k|<s\\a_k\in Z_\epsilon^1}}\log\left|\frac{s^2-\overline{a}_k re^{i\theta}}{s^2-\overline{a}_k (re^{i\theta}+t\omega )}\right|\frac{\mathrm{d}\theta}{2\pi}\\
\leq&\frac{2|\omega |^{\delta/2}}\delta\frac1{1-\delta}\frac1{r^{\delta/2}}\left(n_{\angle\epsilon,\omega }(s,1/g')-n_{\angle\epsilon,\omega }(s-\epsilon,1/g')\right) \\
&+\frac{|\omega |^\delta}\delta\left(\frac2{\sqrt{\epsilon(2-\epsilon)}}\right)^\delta\frac1{1-\delta}\frac1{r^\delta}\left(n(s,1/g')-n(s-\epsilon,1/g')\right)\\
\le&\frac{2}{\delta(1-\delta)}r^{\frac{\delta}{2}(\beta-1)}\left(n_{\angle\epsilon,\omega }(s,1/g')-n_{\angle\epsilon,\omega }(s-\epsilon,1/g')\right)\\
&+\left(\frac2{\sqrt{\epsilon(2-\epsilon)}}\right)^\delta\frac1{\delta(1-\delta)}r^{\delta(\beta-1)}\left(n(s,1/g')-n(s-\epsilon,1/g')\right)
\end{aligned}
\end{equation}
and
\begin{equation}
    \label{eq5.10}
    \begin{aligned}
&\int_0^{2\pi}\max_{t\in[0,1]}\sum_{\substack{|b_k|<s\\b_k\in P_\epsilon^1}}\log\left|\frac{re^{i\theta}-b_k}{re^{i\theta}+t\omega -b_k}\right|\frac{\mathrm{d}\theta}{2\pi}\\
\leq&\frac{2|\omega |^{\delta/2}}{\delta}\frac{1}{1-\delta}\frac{1}{r^{\delta/2}}\left(n_{\angle\epsilon,\omega }(r+|\omega |+\epsilon,g')-n_{\angle\epsilon,\omega }(r-|\omega |-\epsilon,g')\right) \\
&+\frac{|\omega |^\delta}{\delta}\left(\frac{2}{\sqrt{\epsilon(2-\epsilon)}}\right)^\delta\frac{1}{1-\delta}\frac{1}{r^\delta}\left(n(r+|\omega |+\epsilon,g')-n(r-|\omega |-\epsilon,g')\right)\\
\le&\frac{2}{\delta(1-\delta)}r^{\frac{\delta}{2}(\beta-1)}\left(n_{\angle\epsilon,\omega }(r+|\omega |+\epsilon,g')-n_{\angle\epsilon,\omega }(r-|\omega |-\epsilon,g')\right)\\
&+\left(\frac2{\sqrt{\epsilon(2-\epsilon)}}\right)^\delta\frac1{\delta(1-\delta)}r^{\delta(\beta-1)}\left(n(r+|\omega |+\epsilon,g')-n(r-|\omega |-\epsilon,g')\right),
\end{aligned}
\end{equation}
where we extend our exceptional set to include the bounded exceptional sets from our application of \cite[Lemma 3.1]{asikainen2023new}.

We may apply Lemma \ref{l2.3} and Lemma \ref{l2.4} to the unintegrated counting functions $n(\cdot,g')$ and $n(\cdot,1/g')$ (and similarly to $n_{\angle\epsilon,\omega}(\cdot,g')$ and $n_{\angle\epsilon,\omega}(\cdot,1/g')$), since they are increasing functions that can be continuously approximated up to arbitrary precision, and since their hyper-orders are less than or equal to the hyper-order of $g'$, which is less than $1$ by hypothesis.
Therefore, by Lemma \ref{l2.3} and Lemma \ref{l2.4}, we obtain the following estimates:
$$
n(s,1/g')-n(s-\epsilon,1/g')\leq C_1\frac{n(s,1/g')}{r^{1-\varsigma-\tau_1}},
$$
and
$$
n(r+|\omega|+\epsilon,g')-n(r-|\omega|-\epsilon,g')\leq C_2\frac{n(s,g')}{r^{1-\varsigma-\beta-\tau_2}}
$$
where $C_1,C_2>0$, $\tau_1\in(0,1-\varsigma)$ and $\tau_2\in(0,1-\varsigma-\beta)$ are constants. Substituting these in \eqref{eq5.9} and \eqref{eq5.10} yields the following estimates outside some exceptional set of ﬁnite logarithmic measure that depends on $\epsilon$, $\tau_1$, $\tau_2$ and $g$:
\begin{equation}
    \label{eq5.12}
\begin{aligned}
&\int_0^{2\pi}\max_{t\in[0,1]}\sum_{\substack{|a_k|<s\\a_k\in Z_\epsilon^1}}\log\left|\frac{s^2-\overline{a}_k re^{i\theta}}{s^2-\overline{a}_k (re^{i\theta}+t\omega )}\right|\frac{\mathrm{d}\theta}{2\pi} \\
\le &C_{3}\frac{\alpha}{\alpha-1}\left(\frac{1}{r^{\delta(1-\beta)+1-\varsigma-\tau_1}}+\frac{1}{r^{\frac{\delta}{2}(1-\beta)+1-\varsigma-\tau_1}}\frac{n_{\angle\epsilon,\omega }(s,1/g')}{n(s,1/g')}\right)N(\alpha(r+|\omega |),1/g')
\end{aligned}
\end{equation}
and
\begin{equation}
    \label{eq5.13}
    \begin{aligned}&\int_0^{2\pi}\max_{t\in[0,1]}\sum_{|b_k|<s}\log\left|\frac{re^{i\theta}-b_k}{re^{i\theta}+t\omega -b_k}\right|\frac{\mathrm{d}\theta}{2\pi}\\
    \leq &C_3\frac{\alpha}{\alpha-1}\left(\frac{1}{r^{\delta(1-\beta)+1-\varsigma-\beta-\tau_2}}+\frac{1}{r^{\frac{\delta}{2}(1-\beta)+1-\varsigma-\beta-\tau_2}}\frac{n_{\angle\epsilon,\omega }(s,1/g')}{n(s,1/g')}\right)N(\alpha(r+|\omega |),g')\end{aligned}
\end{equation}
where $C_3$ is a sufficiently large constant. Including the resulting exceptional sets in our overall exceptional set, which is still of ﬁnite logarithmic measure, we collect the estimates \eqref{eq5.2}, \eqref{eq5.4}, \eqref{eq5.5}, \eqref{eq5.6}, \eqref{eq5.7}, \eqref{eq5.8}, \eqref{eq5.12} and \eqref{eq5.13}, to obtain
\begin{equation}
    \label{eq5.14}
    \begin{aligned}
    &m\left(r,\frac{\Delta_\omega g(re^{i\theta})}{g'(re^{i\theta})}\right)\\&\leq C_4\frac{\alpha+1}{\alpha-1}\Biggl(\frac{T(\alpha(r+|\omega |),g')}{r^{\delta(1-\beta)}}+\frac{N(\alpha(r+|\omega |),g')}{r^{\delta(1-2\beta)}}\\
    &\quad+\frac{N(\alpha(r+|\omega |),g')}{r^{\delta(1-\beta)+1-\varsigma-\tau_1}}+\frac{N(\alpha(r+|\omega |),g')}{r^{\delta(1-\beta)+1-\varsigma-\beta-\tau_2}}\\
    &\quad+R_{\epsilon,\omega }(s,g')\left(\frac{N(\alpha(r+|\omega |),g')}{r^{\frac{\delta}{2}(1-\beta)+1-\varsigma-\tau_1}}+\frac{N(\alpha(r+|\omega |),g')}{r^{\frac{\delta}{2}(1-\beta)+1-\varsigma-\beta-\tau_2}}\right)\\
    &\quad+R_{\epsilon,\omega }\left(s,\frac{1}{g'}\right)\left(\frac{N\left(\alpha(r+|\omega |),\frac{1}{g'}\right)}{r^{\frac{\delta}{2}(1-\beta)+1-\varsigma-\tau_1}}+\frac{N\left(\alpha(r+|\omega |),\frac{1}{g'}\right)}{r^{\frac{\delta}{2}(1-\beta)+1-\varsigma-\beta-\tau_2}}\right)\Biggr)\\
    &\quad+\beta\log r,\end{aligned}
\end{equation}
where $C_4$ is a sufficiently large constant. Choosing
\begin{equation}
    \label{eq5.15}
    \alpha=1+\frac{r+|\omega |}{(r+|\omega |)(\log T(r+|\omega |,g'))^{1+\tau}},
\end{equation}
where $\tau=\max\{\tau_1,\tau_2\}$, we rid of the coefficient $\alpha$ in terms in \eqref{eq5.14}: indeed with this choice of $\alpha$, by \cite[Lemma 3.3.1]{cherry2013nevanlinna}
\begin{equation}
    \label{eq5.11}
    T(\alpha(r+|\omega|),g')\leq CT(r+|\omega|,g').
\end{equation}
We use the same logic for the counting function terms in \eqref{eq5.14}. Setting $\delta=1-\tau$ in \eqref{eq5.14}, using \eqref{eq5.11} and estimating the terms involving $\alpha$ by \eqref{eq5.15} and the fact that the hyper-order of $g^\prime=f^\prime-a$ is less than $1$, all together yields
\begin{equation}
    \label{eq5.16}
    \begin{aligned}
        m\left(r,\frac{\Delta_\omega f(re^{i\theta})-a\omega}{f^\prime(re^{i\theta})-a}\right)=&O\left(\frac{T(r+|\omega|,f^\prime)}{r^{1-\varsigma-\beta-\tau(1+\varsigma-\beta)}}\right)+O\left(\frac{N(r+|\omega|,f^\prime)}{r^{1-\varsigma-2\beta-\tau(1+\varsigma-2\beta)}}\right)\\
        &+O\left(\frac{N(r+|\omega|,f^\prime)}{r^{2-2\varsigma-2\beta-\tau(2+\varsigma-\beta)}}\right)\\
        &+O\Bigg(R_{\epsilon,\omega }\left(r+|\omega|,{f'}\right)\frac{N\left(r+|\omega |,{f'}\right)}{r^{\frac{3}{2}-2\varsigma-\frac{3}{2}\beta-\tau\left(\frac{3}{2}+\varsigma-\frac{1}{2}\beta\right)}}+\\
        &R_{\epsilon,\omega }\left(r+|\omega|,\frac{1}{f'-a}\right)\frac{N\left(r+|\omega |,\frac{1}{f'-a}\right)}{r^{\frac{3}{2}-2\varsigma-\frac{3}{2}\beta-\tau\left(\frac{3}{2}+\varsigma-\frac{1}{2}\beta\right)}}\Bigg)\\
        &+O(\log r),
    \end{aligned}
\end{equation}
as $r\to\infty$. Further, by Lemma \ref{l2.4} we may also eliminate the shift in argument by $\omega$ in \eqref{eq5.16}:
$$
T(r+|\omega|,f')=T(r,f')+o\left(\frac{T(r,f')}{r^{1-\varsigma-\tau}}\right),
$$
and we similarly handle the shift in the remaining counting function terms in \eqref{eq5.16} (including the ones in $R_{\epsilon,\omega }$-terms). The assertion of the theorem is obtained by setting $\tau=\frac{\epsilon}{3}$ and by including all the aforementioned exceptional sets of finite logarithmic measure in our final exceptional set. 

\section{The proof of Theorem \ref{t.angular}}

Let $s=\frac{\alpha+1}{2} r$, where $\alpha:=\alpha(r)>1$ and $\alpha\to1$ as $r\to\infty$. 
Using the Poisson-Jensen formula for the angular shift $f(e^{i\omega(r)}z)-f(z)$, as previously done in \eqref{eq5.2}, we obtain:
\begin{align*}
&\quad \quad m\left(r,\frac{f(e^{i\omega(r)}z)-f(z)}{f'}\right)\\[6pt]
&\leq\int_0^{2\pi} \sup_{\eta < \omega(r)}\Bigg(\int_0^{2\pi}\log|f'(s e^{i\psi})|\left(P\left(\frac{re^{i\theta+i\eta}}{s},\psi\right)-P\left(\frac{re^{i\theta}}{s},\psi\right)\right)\frac{{\rm d}\psi}{2\pi}\\
&\quad\quad\quad-\sum_{|a_k|< s} G\left(\frac{re^{i\theta+i\eta}}{s},\frac{a_k}{s}\right)+G\left(\frac{re^{i\theta}}{s},\frac{a_k}{s}\right)\\
&\quad\quad\quad+\sum_{|b_k|< s} G\left(\frac{re^{i\theta+i\eta}}{s},\frac{b_k}{s}\right)-G\left(\frac{re^{i\theta}}{s},\frac{b_k}{s}\right)\Bigg)\frac{{\rm d}\theta}{2\pi}+\log r,
\end{align*}
where $\{a_k\}_{k\in\mathbb{N}}$ and $\{b_k\}_{k\in\mathbb{N}}$ are the sequences of the zeros and poles of $f$, respectively. The integral of Poisson kernel's oscillation: we have
\begin{align*}
    \left|P\left(\frac{r}{s}e^{i\theta+i\eta},\psi\right)-P\left(\frac{r}{s}e^{i\theta},\psi\right)\right|=&\left|\operatorname{Re}\left(\frac{s e^{i\psi}+re^{i\theta+i\eta}}{s e^{i\psi}-re^{i\theta+i\eta}}-\frac{s e^{i\psi}+re^{i\theta}}{s e^{i\psi}-re^{i\theta}}\right)\right|\\
    =&\left|\operatorname{Re}\frac{2s re^{i\theta+i\psi}(e^{i\eta}-1)}{(s e^{i\psi}-re^{i\theta+i\eta})(s e^{i\psi}-re^{i\theta})}\right|\\
    \le&\frac{2s r|e^{i\eta}-1|}{|s e^{i\psi}-re^{i\theta+i\eta}||s e^{i\psi}-re^{i\theta}|},
\end{align*}
and we swap the order of integration by Fubini’s theorem to 

\begin{equation}
    \label{estimate. Poisson kernel}
    \begin{aligned}
&\quad\int_0^{2\pi} \sup_{\eta < \omega(r)}\int_0^{2\pi}\log|f'(s e^{i\psi})|\left(P\left(\frac{re^{i\theta+i\eta}}{s},\psi\right)-P\left(\frac{re^{i\theta}}{s},\psi\right)\right)\frac{{\rm d}\psi}{2\pi}\frac{{\rm d}\theta}{2\pi}\\ 
&\leq \int_0^{2\pi} \int_0^{2\pi}\log|f'(s e^{i\psi})|\sup_{\eta < \omega}\left|P\left(\frac{re^{i\theta+i\eta}}{s},\psi\right)-P\left(\frac{re^{i\theta}}{s},\psi\right)\right|\frac{{\rm d}\psi}{2\pi}\frac{{\rm d}\theta}{2\pi}\\ 
&= \int_0^{2\pi} \log|f'(s e^{i\psi})|\int_0^{2\pi}\sup_{\eta < \omega}\left|P\left(\frac{re^{i\theta+i\eta}}{s},\psi\right)-P\left(\frac{re^{i\theta}}{s},\psi\right)\right|\frac{{\rm d}\theta}{2\pi}\frac{{\rm d}\psi}{2\pi}\\ 
&\leq 2%(s^2-r^2)
s r|e^{i\omega(r)}-1|\\
&\quad\cdot\int_0^{2\pi} \log|f'(s e^{i\psi})|\int_0^{2\pi}\sup_{\eta < \omega}\frac{1}{|s e^{i\psi}-re^{i\theta+i\eta}||s e^{i\psi}-re^{i\theta}|}\frac{{\rm d}\theta}{2\pi}\frac{{\rm d}\psi}{2\pi}\\ 
&\leq \frac{2%(s^2-r^2)
s r|e^{i\omega(r)}-1|}{(s-r)^{2-\delta}}\int_0^{2\pi} \log|f'(s e^{i\psi})|\int_0^{2\pi}\frac{1}{|s e^{i\psi}-re^{i\theta}|^\delta}\frac{{\rm d}\theta}{2\pi}\frac{{\rm d}\psi}{2\pi}\\ 
&\leq \frac{2%(s^2-r^2)
s r\omega(r)}{(s-r)^{2-\delta}r^\delta(1-\delta)}\left(m(s, f')+m\left(s, \frac1{f'}\right)\right)\\
&\leq\frac{2^{2-\delta}(\alpha+1)\omega(r)}{(\alpha-1)^{2-\delta}(1-\delta)}\left(m(\alpha r, f')+m\left(\alpha r, \frac1{f'}\right)\right).%\\
%&\leq C\frac{T(\alpha r,f^\prime)}{1/\omega(r)}.%\le C\frac{T(\alpha r,f^\prime)}{r}.
\end{aligned}
\end{equation}

The oscillation of the Green function about $a_k$ satisfies:

\begin{align*}
&G\left(\frac{re^{i\theta}}{s},\frac{a_k}{s}\right)-G\left(\frac{re^{i\theta+i\eta}}{s},\frac{a_k}{s}\right)\\ 
&=\log\left|\frac{s-\overline{a}_k \frac{re^{i\theta}}{s}}{re^{i\theta}-a_k}\right|-\log\left|\frac{s-\overline{a}_k \frac{re^{i\theta+i\eta}}{s}}{re^{i\theta+i\eta}-a_k}\right|\\ 
&=\log\left|\frac{s^2-\overline{a}_k re^{i\theta}}{s^2-\overline{a}_k re^{i\theta+i\eta}}\right|+\log\left|\frac{re^{i\theta+i\eta}-a_k}{re^{i\theta}-a_k}\right|.
\end{align*}

Integrating the supremum of the first term:
\begin{align*}
&\quad\int_0^{2\pi}\sup_{\eta<\omega(r)}\log\left|\frac{s^2-\overline{a}_k re^{i\theta}}{s^2-\overline{a}_k re^{i\theta+i\eta}}\right|\frac{{\rm d}\theta}{2\pi}\\ 
&\leq\frac{1}{\delta}\int_0^{2\pi}\sup_{\eta<\omega(r)}\log\left(1+\left|\frac{re^{i\theta+i\eta}-re^{i\theta}}{re^{i\theta+i\eta}-s^2/\overline{a}_k}\right|^\delta\right)\frac{{\rm d}\theta}{2\pi}\\ 
&\leq\frac{r^\delta|e^{i\omega(r)}-1|^\delta}{\delta}\int_0^{2\pi}\sup_{\eta<\omega(r)}\frac{1}{|re^{i\theta+i\eta}-s^2/\overline{a}_k|^\delta}\frac{{\rm d}\theta}{2\pi}\\
&\le\frac{r^\delta\omega(r)^\delta}{\delta}\int_0^{2\pi}\sup_{\eta<\omega(r)}\frac{1}{|re^{i\theta+i\eta}-s^2/\overline{a}_k|^\delta}\frac{{\rm d}\theta}{2\pi}
\end{align*}

Via rotation, we may assume $a_k\in\mathbb{R}^+$:
\begin{align*}
&\int_0^{2\pi}\sup_{\eta<\omega(r)}\frac{1}{|re^{i\theta+i\eta}-s^2/a_k|^\delta}\frac{{\rm d}\theta}{2\pi}\\
= &\frac{1}{2\pi}\int_0^{\pi-\frac{\omega(r)}{2}}\frac{{\rm d}\theta}{|re^{i\theta}-s^2/a_k|^\delta}+\frac{1}{2\pi}\int_{\pi-\frac{\omega(r)}{2}}^{2\pi-\omega(r)}\frac{{\rm d}\theta}{|re^{i\theta+i\omega(r)}-s^2/a_k|^\delta}\\
&+\frac{1}{2\pi}\int_{2\pi-\omega(r)}^{2\pi}\frac{{\rm d}\theta}{|r-s^2/a_k|^\delta}\\
\leq&\int_{0}^{2\pi}\frac{1}{|re^{i\theta}-s^2/a_k|^\delta}\frac{{\rm d}\theta}{2\pi}+\int_{2\pi-\omega(r)}^{2\pi}\frac{1}{|r-s^2/a_k|^\delta}\frac{{\rm d}\theta}{2\pi}\\
\leq&\frac{1}{r^\delta(1-\delta)}+\frac{\omega(r)}{2\pi|r-s^2/a_k|^\delta}\\
\leq&\frac{1}{r^\delta(1-\delta)}+\frac{2^{\delta-1}\omega(r)}{\pi(\alpha-1)^\delta r^\delta}.
\end{align*}

Thus, 
\begin{align*}
    \quad\int_0^{2\pi}\sup_{\eta<\omega(r)}\log\left|\frac{s^2-\overline{a}_k re^{i\theta}}{s^2-\overline{a}_k re^{i\theta+i\eta}}\right|\frac{{\rm d}\theta}{2\pi}\le\frac{\omega^\delta(r)}{\delta(1-\delta)}+\frac{2^{\delta-1}\omega^{1+\delta}(r)}{\pi\delta(\alpha-1)^\delta}.
\end{align*}

On the other hand, we can calculate that
\begin{align*}
\int_0^{2\pi}\sup_{\eta<\omega(r)}\log\left|\frac{re^{i\theta+i\eta}-a_k}{re^{i\theta}-a_k}\right|\frac{{\rm d}\theta}{2\pi}
&\leq\frac{1}{\delta}\int_0^{2\pi}\sup_{\eta<\omega(r)}\log\left(1+\left|\frac{re^{i\theta+i\eta}-re^{i\theta}}{re^{i\theta}-a_k }\right|^\delta\right)\frac{{\rm d}\theta}{2\pi}\\
&\leq\frac{r^\delta|e^{i\omega(r)}-1|^\delta}{\delta}\int_0^{2\pi}\sup_{\eta<\omega(r)}\frac{1}{\left|re^{i\theta}-a_k\right|^{\delta}}\frac{{\rm d}\theta}{2\pi}\\
&\leq\frac{\omega^\delta(r)}{\delta(1-\delta)}.
\end{align*}

Therefore,
\begin{equation}
    \begin{aligned}
    \label{eq6.1}
    &\int_0^{2\pi}\sup_{\eta<\omega(r)}\left(G\left(\frac{re^{i\theta}}{s},\frac{a_k}{s}\right)-G\left(\frac{re^{i\theta+i\eta}}{s},\frac{a_k}{s}\right)\right)\frac{{\rm d}\theta}{2\pi}\\
    \leq&\frac{2\omega^\delta(r)}{\delta(1-\delta)
    }+\frac{2^{\delta-1}\omega^{1+\delta}(r)}{\pi\delta (\alpha-1)^\delta}.
\end{aligned}
\end{equation}

The oscillation of the Green function about $b_k$ satisfies:
\begin{align*}
&G\left(\frac{re^{i\theta+i\eta}}{s},\frac{b_k}{s}\right)-G\left(\frac{re^{i\theta}}{s},\frac{b_k}{s}\right)\\ 
&=\log\left|\frac{s-\overline{b}_k \frac{re^{i\theta+i\eta}}{s}}{re^{i\theta+i\eta}-b_k}\right|-\log\left|\frac{s-\overline{b}_k \frac{re^{i\theta}}{s}}{re^{i\theta}-b_k}\right|\\ 
&=\log\left|\frac{re^{i\theta}-b_k}{re^{i\theta+i\eta}-b_k}\right|+\log\left|\frac{s^2-\overline{b}_k re^{i\theta+i\eta}}{s^2-\overline{b}_k re^{i\theta}}\right|.
\end{align*}

Integrating the supremum of the first term:

\begin{align*}
&\quad\int_0^{2\pi}\sup_{\eta<\omega(r)}\log\left|\frac{re^{i\theta}-b_k}{re^{i\theta+i\eta}-b_k}\right|\frac{{\rm d}\theta}{2\pi}\\ 
&\leq\frac{1}{\delta}\int_0^{2\pi}\sup_{\eta<\omega(r)}\log\left(1+\left|\frac{re^{i\theta+i\eta}-re^{i\theta}}{re^{i\theta+i\eta}-b_k}\right|^\delta\right)\frac{{\rm d}\theta}{2\pi}\\ 
&\leq\frac{r^\delta|e^{i\omega(r)}-1|^\delta}{\delta}\int_0^{2\pi}\sup_{\eta<\omega(r)}\frac{1}{|re^{i\theta+i\eta}-b_k|^\delta}\frac{{\rm d}\theta}{2\pi}\\
&\le\frac{r^\delta\omega(r)^\delta}{\delta}\int_0^{2\pi}\sup_{\eta<\omega(r)}\frac{1}{|re^{i\theta+i\eta}-b_k|^\delta}\frac{{\rm d}\theta}{2\pi}
\end{align*}

Via rotation, we may assume $b_k\in\mathbb{R}^+$:
\begin{align*}
&\int_0^{2\pi}\sup_{\eta<\omega(r)}\frac{1}{|re^{i\theta+i\eta}-b_k|^\delta}\frac{{\rm d}\theta}{2\pi}\\
= &\frac{1}{2\pi}\left(\int_0^{\pi-\frac{\omega(r)}{2}}\frac{{\rm d}\theta}{|re^{i\theta}-b_k|^\delta}+\int_{\pi-\frac{\omega(r)}{2}}^{2\pi-\omega(r)}\frac{{\rm d}\theta}{|re^{i\theta+i\omega(r)}-b_k|^\delta}+\int_{2\pi-\omega(r)}^{2\pi}\frac{{\rm d}\theta}{|r-b_k|^\delta}\right)\\
\leq&\int_{0}^{2\pi}\frac{1}{|re^{i\theta}-b_k|^\delta}\frac{{\rm d}\theta}{2\pi}+\int_{2\pi-\omega(r)}^{2\pi}\frac{1}{|r-b_k|^\delta}\frac{{\rm d}\theta}{2\pi}\\
\leq&\frac{1}{r^\delta(1-\delta)}+\frac{\omega(r)}{2\pi|r-b_k|^\delta}.
\end{align*}
Thus,
\begin{align*}
    \int_0^{2\pi}\sup_{\eta<\omega(r)}\log\left|\frac{re^{i\theta}-b_k}{re^{i\theta+i\eta}-b_k}\right|\frac{{\rm d}\theta}{2\pi}
    \le&\frac{\omega^\delta(r)}{\delta(1-\delta)}+\frac{\omega^{1+\delta}(r)r^\delta}{2\pi\delta(r-|b_k|)^\delta}.
\end{align*}

On the other hand, we can calculate that
\begin{align*}
\int_0^{2\pi}\sup_{\eta<\omega(r)}\log\left|\frac{s^2-\overline{b}_k re^{i\theta+i\eta}}{s^2-\overline{b}_k re^{i\theta}}\right|\frac{{\rm d}\theta}{2\pi}
&\leq\frac{1}{\delta}\int_0^{2\pi}\sup_{\eta<\omega(r)}\log\left(1+\left|\frac{re^{i\theta+i\eta}-re^{i\theta}}{re^{i\theta}-s^2/\overline{b}_k }\right|^\delta\right)\frac{{\rm d}\theta}{2\pi}\\
&\leq\frac{r^\delta|e^{i\omega(r)}-1|^\delta}{s^{2\delta}\delta}\int_0^{2\pi}\sup_{\eta<\omega(r)}\left|\frac{r}{s^2}e^{i\theta}-\frac{1}{\overline{b}_k}\right|^{-\delta}\frac{{\rm d}\theta}{2\pi}\\
&\leq\frac{\omega^\delta(r)}{\delta(1-\delta)}.
\end{align*}

Similarly, we can write
\begin{equation}
    \begin{aligned}
    \label{eq6.2}
    &\int_0^{2\pi}\sup_{\eta<\omega(r)}\left(G\left(\frac{re^{i\theta+i\eta}}{s},\frac{b_k}{s}\right)-G\left(\frac{re^{i\theta}}{s},\frac{b_k}{s}\right)\right)\frac{{\rm d}\theta}{2\pi}\\
    \leq&\frac{2\omega^\delta(r)}{\delta(1-\delta)
    }+\frac{\omega^{1+\delta}(r)r^\delta}{2\pi\delta (r-|b_k|)^\delta}.
\end{aligned}
\end{equation}

The remaining task is to address the estimates for the sum
$\sum_{|b_k|< s}\frac{1}{(r-|b_k|)^\delta}$.

Since there are only finitely many zeros or poles of $f$ whose moduli are close to a chosen value $r$, we can extract the corresponding closed intervals related to $r$, such as $\left[|b_k|-k^{-2},|b_k|+k^{-2}\right]$ for each pole $b_k$. In this case, we have $r-|b_k|>k^{-2}$. 
If we choose $\delta=\frac{\varepsilon}{1+\varepsilon}$, $\alpha=1+\frac{1}{(\log T(r,f))^{1+\varepsilon}}$, and consider the restriction on $\omega(r)$, then by \eqref{eq6.2} and \cite[Lemma 3.3.1]{cherry2013nevanlinna}, we can obtain that the following inequality

\begin{equation}
    \label{estimate green b}
    \begin{aligned}
    &\sum_{|b_k|<s}\int_0^{2\pi}\sup_{\eta<\omega(r)}\left(G\left(\frac{re^{i\theta+i\eta}}{s},\frac{b_k}{s}\right)-G\left(\frac{re^{i\theta}}{s},\frac{b_k}{s}\right)\right)\frac{{\rm d}\theta}{2\pi}\\[6pt]
    &\leq\sum_{|b_k|<s}\frac{2\omega^\delta(r)}{\delta(1-\delta)%r^\delta
    }+\frac{\omega^{1+\delta}(r)r^\delta}{2\pi\delta (r-|b_k|)^\delta}\\[6pt]
    &\leq\sum_{|b_k|<s}\frac{2\omega^\delta(r)}{\delta(1-\delta)%r^\delta
    }+\frac{\omega^{1+\delta}(r)r^\delta k^{2\delta}}{2\pi\delta}\\[6pt]
    &\leq C_1\frac{\omega^{1+\delta}(r)r^\delta n(s,f)^{2\delta}}{\delta}n(s,f)\\[6pt]
    &\leq 2C_1\frac{\omega^{1+\delta}(r)r^\delta n(s,f)^{2\delta}}{\delta(\alpha-1)}T(\alpha r,f)\\[6pt]
    &\leq 2C_1\frac{r^\delta (\log T(r,f))^{1+\varepsilon}T(\alpha r,f)}{\delta%(\alpha-1)
    r^{\delta(1+\delta)}% + (2-\delta)\varsigma(1+\epsilon)
    T(r+\varepsilon,f^\prime)^{\frac{\varepsilon}{1+\varepsilon}}}\\[6pt]
    %&\leq C_2\frac{r^{\varsigma(1+\epsilon)}}{r^{\delta^2+\varsigma(2-\delta)(1+\varepsilon)}}T(r,f)\\[6pt]
    %&= C_2\frac{T(r,f)}{r^{\delta^2+\varsigma(1-\delta)(1+\epsilon)}}=
    &\leq C_2\frac{T(r,f)}{r^{\frac{\varepsilon^2}{(1+\varepsilon)^2}%+\varsigma\varepsilon
    }}=S(r,f)
    \end{aligned}
\end{equation}
holds for all $r$ outside a set of finite linear measure where $C_1$ and $C_2$ are positive constants, and since we can consider only $r\in[1,\infty)$, the exceptional set also possesses finite logarithmic measure). Similarly, from \eqref{eq6.1} and \eqref{estimate. Poisson kernel}, we find that
\begin{equation}
    \label{Eq6.1}
    \begin{aligned}
        &\sum_{|a_k|<s}\int_0^{2\pi}\sup_{\eta<\omega(r)}\left(G\left(\frac{re^{i\theta}}{s},\frac{a_k}{s}\right)-G\left(\frac{re^{i\theta+i\eta}}{s},\frac{a_k}{s}\right)\right)\frac{{\rm d}\theta}{2\pi}\\
        \le&\sum_{|a_k|<s}\frac{2\omega^\delta(r)}{\delta(1-\delta)%r^\delta
        }+\frac{2^{\delta-1}\omega^{1+\delta}(r)}{\pi\delta (\alpha-1)^\delta}\\
        \le&C_3\frac{\omega^{1+\delta}(r)}{\delta(\alpha-1)^{1+\delta}}T(\alpha r,f)\\
        \le&C_4\frac{T(r,f)}{r^{\delta(1+\delta)%+(1-\delta^2)(1+\varepsilon)\varsigma
        }}\\
        =&S(r,f),
    \end{aligned}
\end{equation}
and
\begin{equation}
    \label{Eq6.2}
    \begin{aligned}
        &\int_0^{2\pi} \sup_{\eta < \omega(r)}\int_0^{2\pi}\log|f'(s e^{i\psi})|\left(P\left(\frac{re^{i\theta+i\eta}}{s},\psi\right)-P\left(\frac{re^{i\theta}}{s},\psi\right)\right)\frac{{\rm d}\psi}{2\pi}\frac{{\rm d}\theta}{2\pi}\\
        \leq& C_5\frac{T(r,f')}{r^{\frac{\varepsilon}{1+\varepsilon}}}=S(r,f),
    \end{aligned}
\end{equation}
hold for all $r$ outside a set of finite linear measure, respectively, where $C_3$, $C_4$, and $C_5$ are positive constants.

By \eqref{estimate green b}, \eqref{Eq6.1} and \eqref{Eq6.2}, we can obtain the following estimate,
$$m\left(r,\frac{f(e^{i\omega(r)}z)-f(z)}{f'}\right)=S(r,f)$$

holds for all $r$ outside an exceptional set with finite logarithmic measure.

\section{Consequences}

\subsection{Deficiencies of meromorphic functions}

\begin{proposition}
    \label{p5.1} 
    Let $f$ be a transcendental meromorphic function of finite order that is non-constant. And let $r=|z|$ such that $0<|\eta|<\alpha_1(r)$, where
    $$
    \alpha_1(r)=\min\left\{r,\log^{-\frac{1}{2}}r,\frac h2,\frac1{\sum_{0<|b_\mu|<r+\frac12}1/|b_\mu|}\right\},
    $$
    where $(b_\mu)_{\mu\in\mathbb{N}}$ is the sequence of poles of $f (z)$, and let $h \in (0,1)$ be such that $f (z)$ has no poles in $\overline{D}(0,h)\setminus\{0\}$. Then
    $$
    \delta(a,f')\leq\left(1+\limsup_{r\to\infty}\frac{N(r,f)}{T(r,f')}\right)\delta(a\eta,\Delta_\eta f).
    $$
    In particular, when $f$ is an entire function, we have
    $$
    \delta(a,f')\leq\delta\left(a,\frac{\Delta_\eta f}\eta\right).
    $$
\end{proposition}

\begin{proof}
    With the notation 
    \begin{equation}
        \label{e5.2}
        A(a,r)=m\left(r,\frac{\Delta_\eta f-a\eta}{f^\prime-a}\right)
    \end{equation}
we have
$$
m\left(r,\frac{1}{f'-a}\right)\leq m\left(r,\frac{1}{\Delta_\eta f-a\eta }\right)+A(a,r)\quad\text{and}\quad m(r,\Delta_\eta f)\leq m(r,f')+A(0,r).
$$
Thus, for all large enough $r>0$
$$
\begin{aligned}
\frac{T(r,\Delta_\eta f)}{T(r,f')}&\leq\frac{m(r,f')+N(r,f)+N_\eta(r,f(z+\eta ))}{T(r,f')}+\frac{A(0,r)}{T(r,f')} \\
&\leq1+\frac{N(r,f)}{T(r,f')}+\frac{O(\log r)}{T(r,f')}+\frac{A(0,r)}{T(r,f')},
\end{aligned}
$$
where we have used the following estimate \cite[Theorem 2.3]{MR3695346}
$$
N_\eta(r,f(z+\eta))=N(r,f(z))+\varepsilon_1(r),
$$
where $|\varepsilon_1(r)|\leq n(0,f(z))\log r+3$. Thus,
$$
\begin{aligned}
\liminf_{r\to\infty}\frac{m\left(r,\frac{1}{f'-a}\right)}{T(r,f')}&\leq\liminf_{r\to\infty}\left(\frac{m\left(r,\frac{1}{\Delta_\eta f-a\eta }\right)}{T(r,\Delta_\eta f)}\frac{T(r,\Delta_\eta f)}{T(r,f')}+\frac{A(a,r)}{T(r,f')}\right) \\
&\leq\liminf_{r\to\infty}\Bigg[\left(1+\frac{N(r,f)}{T(r,f')}+\frac{O(\log r)+A(0,r)}{T(r,f')}\right) \\
&\quad\times\frac{m\left(r,\frac{1}{\Delta_\eta f-a\eta }\right)}{T(r,\Delta_\eta f)}+\frac{A(a,r)}{T(r,f')}\Bigg].
\end{aligned}
$$
By using the result of Theorem \ref{t1.1}, we immediately get the ﬁrst statement of the proposition, and the second statement is an immediate consequence of the first.
\end{proof}

\begin{proposition}
    \label{p5.2}
    Let $0<\beta<\frac{1}{2}$, $0<|\omega(r)|<r^\beta$, and let $f$ be a transcendental meromorphic function of finite order $\rho<\infty$ that is not $\omega$-periodic. If the lower order $\mu(f^\prime)$ of $f^\prime$ satisfies $\mu(f^\prime)>\rho(f)-\frac{1}{2}$, then
    $$
    \delta(a,f')\leq\left(1+\limsup_{r\to\infty}\frac{N(r,f)}{T(r,f')}\right)\delta(a\omega,\Delta_\omega f).
    $$
    In particular, when $f$ is an entire function, we have
    $$
    \delta(a,f')\leq\delta\left(a,\frac{\Delta_\omega f}\omega\right).
    $$
\end{proposition}

\begin{proof}
    Just like the proof of Proposition \ref{p5.1} above, with the notation
    \begin{equation}
        \label{e5.2*}
        A_*(a,r)=m\left(r,\frac{\Delta_\omega f-a\omega}{f^\prime-a}\right)
    \end{equation}
    we have
$$
m\left(r,\frac{1}{f'-a}\right)\leq m\left(r,\frac{1}{\Delta_\omega f-a\omega }\right)+A(a,r)\ \text{and}\  m(r,\Delta_\omega f)\leq m(r,f')+A(0,r).
$$
Similarly, for all large enough $r>0$,
$$
\begin{aligned}
\frac{T(r,\Delta_\omega f)}{T(r,f')}&\leq\frac{m(r,f')+N(r,f)+N_\eta(r,f(z+\omega ))}{T(r,f')}+\frac{A_*(0,r)}{T(r,f')} \\
&\leq1+\frac{N(r,f)}{T(r,f')}+\frac{O(r^{\max\{\rho(f^\prime)-1,0\}+\beta+\epsilon})+O(\log r)}{T(r,f')}+\frac{A_*(0,r)}{T(r,f')},
\end{aligned}
$$
where we have used the following estimate \cite[Theorem 3.2]{MR3695346}
$$
N(r,f(z+\omega))=N(r,f)+O(r^{\max\{\rho(f^\prime)-1,0\}+\beta+\epsilon})+O(\log r).
$$
Thus,
$$
\begin{aligned}
\liminf_{r\to\infty}\frac{m\left(r,\frac{1}{f'-a}\right)}{T(r,f')}&\leq\liminf_{r\to\infty}\left(\frac{m\left(r,\frac{1}{\Delta_\omega f-a\omega }\right)}{T(r,\Delta_\omega f)}\frac{T(r,\Delta_\omega f)}{T(r,f')}+\frac{A_*(a,r)}{T(r,f')}\right) \\
&\leq\liminf_{r\to\infty}\Bigg[\Bigg(1+\frac{N(r,f)}{T(r,f')}\\
&\quad+\frac{O(r^{\max\{\rho(f^\prime)-1,0\}+\beta+\epsilon})+O(\log r)+A_*(0,r)}{T(r,f')}\Bigg) \\
&\quad\times\frac{m\left(r,\frac{1}{\Delta_\omega f-a\omega }\right)}{T(r,\Delta_\omega f)}+\frac{A_*(a,r)}{T(r,f')}\Bigg].
\end{aligned}
$$
By using the result of Theorem \ref{t1.2} and the condition ``$\mu(f^\prime)>\rho(f)-\frac{1}{2}$", we immediately get the ﬁrst statement of the proposition, and the second statement is an immediate consequence of the first.
\end{proof}

\medskip

\subsection{$\eta/\omega$-separated pair indices of entire functions with finite order}
~

Now we recall here the definition of the $c$-separated pair index $\pi_c(a,f)$ for the value $a\in\mathbb{C}\bigcup\{\infty\}$ of a meromorphic function $f$ from \cite{MR2248826}.

\begin{definition}[$c$-separated pair indices]
    \label{d5.2}
    Let $f$ be a meromorphic function that is not $c$-periodic. %Let $0<|\eta|<\alpha(r)$ and $0<|\omega|<r^\beta$, where $\alpha(r)$ is an arbitrary term such that $\alpha(r)\to0$ as $r\to\infty$, and $\beta<\frac{1}{2}$. 
    Then the \textbf{$c$-separated pair index} of $a\in\mathbb{C}\bigcup\{\infty\}$ is defined as
    $$
    \pi_c(a,f)=\liminf_{r\to\infty}\frac{N_c(r,a,f)}{T(r,f)},%\ \text{and}\ \pi_\omega(a,f)=\liminf_{r\to\infty}\frac{N_\omega(r,a,f)}{T(r,f)},
    $$
    where the counting function
    $$
    N_c(r,a,f)=\int_0^r\frac{n_c(t,a,f)-n_c(0,a,f)}{t}\mathrm{d}t+n_c(0,a,f)\log r
    $$
    is defined such that $n_c(r, a,f)$ %(where $*$ is $\eta$ or $\omega$) 
    counts a point $|z_0|\le r$ with $f(z_0)=a=f(z_0+c)$ (such points are called $c$-separated $a$-pairs) according to the number of equal terms at the beginning of the Taylor expansions of $f$ at $z_0$ and at $z_0+c$. For poles $n_c(r,f)$ counts the $c$-separated $0$-pairs of $1/f$.
\end{definition}

By replacing \( c \) with \( \eta \) or \( \omega \), as defined in Theorems \ref{t1.1} and \ref{t1.2}, respectively, we obtain the deficiency relations associated with the \( \eta/\omega \)-separated pair indices in the case of finite-order entire functions.

\begin{proposition}
    \label{p5.3}
    Let $\eta$ be the same as defined in Theorem \ref{t1.1}. Next, we substitute $c$ with $\eta$ in Definition \ref{d5.2}, and assume that $f$ is a non-constant entire function of finite order. Then we have
    $$
    \sum_{a\in\mathbb{C}}\pi_\eta(a,f)\leq1-\delta(0,f^{\prime}).
    $$
\end{proposition}

\begin{proof}
    Applying the continuous variable variant of Fatou’s lemma with the counting measure, we have
    $$
    \sum_{a\in\mathbb{C}}\pi_\eta(a,f)=\sum_{a\in\mathbb{C}}\liminf_{r\to\infty}\frac{N_\eta(r,a,f)}{T(r,f)}\leq\liminf_{r\to\infty}\sum_{a\in\mathbb{C}}\frac{N_\eta(r,a,f)}{T(r,f)}.
    $$
    Let $A$ be defined the same as in \eqref{e5.2}, so that we obtain
    $$
    \begin{aligned}
&\sum_{a\in\mathbb C}\pi_\eta(a,f)+\Theta(0,f')\leq\liminf_{r\to\infty}\sum_{a\in\mathbb C}\frac{N_\eta(r,a,f)}{T(r,f)}-\lim\sup_{r\to\infty}\frac{N(r,1/f')}{T(r,f')}+1 \\
&\leq\liminf_{r\to\infty}\frac{N(r,1/\Delta_\eta f)}{T(r,f')}\frac{m(r,f')}{m(r,f)}+\liminf_{r\to\infty}\frac{-\overline{N}(r,1/f')}{T(r,f')}+1 \\
&\leq\liminf_{r\to\infty}\frac{N(r,1/\Delta_\eta f)}{T(r,f')}+\liminf_{r\to\infty}\frac{-\overline{N}(r,1/f')}{T(r,f')}+1\\
&\leq\liminf_{r\to\infty}\frac{\int_0^{2\pi}\log\left|\Delta_\eta f\left(re^{i\theta}\right)\right|\frac{\mathrm{d}\theta}{2\pi}+O(1)}{T(r,f')}+\liminf_{r\to\infty}\frac{-\overline{N}(r,1/f')}{T(r,f')}+1\\
&=\liminf_{r\to\infty}\frac{\int_0^{2\pi}\log\left|f^\prime\left(re^{i\theta}\right)\right|\frac{\mathrm{d}\theta}{2\pi}+\int_0^{2\pi}\log\left|\frac{\Delta_\eta f\left(re^{i\theta}\right)}{f^\prime\left(re^{i\theta}\right)}\right|\frac{\mathrm{d}\theta}{2\pi}}{T(r,f')}+\liminf_{r\to\infty}\frac{-\overline{N}(r,1/f')}{T(r,f')}+1\\
&\leq\lim\inf_{r\to\infty}\frac{N(r,1/f')-\overline{N}(r,1/f')+A(0,r)}{T(r,f')}+1=\theta(0,f')+1,\end{aligned}
    $$
    from which the statement immediately follows.
\end{proof}

\begin{proposition}
    \label{p5.4}
    Let $\omega$ be the same as be the same as defined in Theorem \ref{t1.2}. Next, we substitute $c$ with $\omega$ in Definition \ref{d5.2}, and let $f$ be a non-constant entire function of finite order. Then we have
    $$
    \sum_{a\in\mathbb{C}}\pi_\omega(a,f)\leq1-\delta(0,f^{\prime}).
    $$
\end{proposition}

The proof of Proposition \ref{p5.4} is almost identical to that of Proposition \ref{p5.3} above, so we will omit the details.

\medskip

%\iffalse
\subsection{Analogue of 2nd main theorem}
~

\begin{proposition}
    \label{Paper 3 p5}
    Assume that a meromorphic function $f$, satisfying the conditions of Theorem \ref{t1.1}, possesses a meromorphic primitive $F$. Let $a_j$ be distinct linear functions of $z$, and suppose their derivatives $a_j'$ are also distinct. Under these conditions, we obtain the following estimate:
$$
\begin{aligned}
    \sum_{j=1}^qm\left(r,\frac{1}{f-a_j^\prime}\right)&\le% m\left(r, \frac{1}{\Delta_\eta^2F}\right)+m\left(r,\sum_{j=1}^q\frac{\Delta_\eta^2F}{f-a_j^\prime}\right)+O(1)\\
    %&\le T(r,\Delta_\eta^2F)-N\left(r, \frac{1}{\Delta_\eta^2F}\right)+S(r,f)\\
    %&\le m(r,f)+m\left(r, \frac{\Delta_\eta^2F}{f}\right)+N(r,\Delta_\eta^2F)\\
    %&\quad-N\left(r, \frac{1}{\Delta_\eta^2F}\right)+S(r,f)\\
    %&\le m(r,f)+m\left(r, \frac{\Delta_\eta^2F}{f^\prime}\right)+m\left(r, \frac{f^\prime}{f^\prime}\right)+N(r,\Delta_\eta^2F)\\
    %&\quad-N\left(r, \frac{1}{\Delta_\eta^2F}\right)+S(r,f)\\
    %&=
    m(r,f)+N(r,\Delta_\eta^2F)-N\left(r, \frac{1}{\Delta_\eta^2F}\right)+S(r,f),
\end{aligned}
$$
outside an exceptional set of finite logarithmic measure.
\end{proposition}

\begin{proof}
    From the proof of Theorem \ref{t1.1}, we get for meromorphic, non-constant $g$ for sufficiently large $r>0$,
\begin{align*}
m\left(r,\frac{\Delta_\eta g}{g'}\right)
\leq& \frac{4|\eta|^\delta(\alpha+1)}{(1-\delta)(\alpha-1)r^\delta}(T(\alpha(r+|\eta|),g')+O(1))\\
&+\frac{4|\eta|^{\delta}}{\delta(1-\delta)}\frac{\alpha}{\alpha-1}\frac{1}{r^\delta}T(\alpha(r+|\eta|),g')\\
&+\left(\frac{|\eta|+\epsilon}{\epsilon}\right)^{\delta}\frac{4|\eta|^\delta}{\delta(1-\delta)}\frac{\alpha}{\alpha-1}\frac{1}{r^\delta}T(\alpha(r+|\eta|),g')\\
&+\frac{2|\eta|^{\delta/2}}{\delta}\frac{1}{1-\delta}\frac{1}{r^{\delta/2}}\biggl(n(s,g')+n(s,1/g')\biggl)\\
&+\frac{|\eta|^\delta}{\delta}\left(\frac{2}{\sqrt{\epsilon(2-\epsilon)}}\right)^\delta\frac{1}{1-\delta}\frac{1}{r^{\delta}}\biggl(n(s,g')+n(s,1/g')\biggl)\\
&\leq\frac{4|\eta|^\delta(\alpha+1)}{\delta(1-\delta)(\alpha-1)r^{\frac{\delta}{2}}}\left(5+\left(\frac{|\eta|+\epsilon}{\epsilon}\right)^{\delta}\right)\biggl(T(\alpha(r+|\eta|),g')+O(1)\biggl)\\
&\leq\frac{64|\eta|^\delta(\alpha+1)}{\delta(1-\delta)(\alpha-1)r^{\delta/2}}T(\alpha(r+|\eta|),g')
\end{align*}

If $\alpha=1+|\eta|^\delta$ and 
$$|\eta|<\frac{1}{2rT(r,g')},$$ 
then by Borel lemma for sufficiently large $r$ outside an exceptional set of finite linear measure, we obtain
$$m\left(r,\frac{\Delta_\eta f - a\eta}{f'-a}\right)\leq\frac{512}{\delta(1-\delta)}\frac{T(r,f')}{r^{\delta/2}}=S(r,f^\prime).$$
Next, we can get the following estimate similar to \cite[Corollary 2.5]{asikainen2023new},
$$
m\left(r,\frac{\Delta_\eta^nf-a\eta^n}{f^{(n)}-a}\right)=S(r,f^\prime).
$$
Moreover, if $\Delta_\eta^{n+k-1}f$ is non-constant for $k\in\mathbb{N}$, then
$$
m\left(r,\frac{\Delta_\eta^{n+k-1}f}{f^{(n)}}\right)=S(r,f^\prime).
$$
The proof only requires a small change of signs, so we omit it here.

Therefore, if we assume that a meromorphic function $f$ satisfying the conditions of Theorem \ref{t1.1} possesses a meromorphic primitive $F$, $a_j$ are distinct linear functions of $z$ and their derivatives are also distinct, we will obtain
$$
\begin{aligned}
    \sum_{j=1}^qm\left(r,\frac{1}{f-a_j^\prime}\right)&\le m\left(r, \frac{1}{\Delta_\eta^2F}\right)+m\left(r,\sum_{j=1}^q\frac{\Delta_\eta^2F}{f-a_j^\prime}\right)+O(1)\\
    &\le T(r,\Delta_\eta^2F)-N\left(r, \frac{1}{\Delta_\eta^2F}\right)+S(r,f)\\
    &\le m(r,f)+m\left(r, \frac{\Delta_\eta^2F}{f}\right)+N(r,\Delta_\eta^2F)\\
    &\quad-N\left(r, \frac{1}{\Delta_\eta^2F}\right)+S(r,f)\\
    %&\le m(r,f)+m\left(r, \frac{\Delta_\eta^2F}{f^\prime}\right)+m\left(r, \frac{f^\prime}{f^\prime}\right)+N(r,\Delta_\eta^2F)\\
    %&\quad-N\left(r, \frac{1}{\Delta_\eta^2F}\right)+S(r,f)\\
    &=m(r,f)+N(r,\Delta_\eta^2F)-N\left(r, \frac{1}{\Delta_\eta^2F}\right)+S(r,f),
\end{aligned}
$$
outside an exceptional set of finite logarithmic measure.
\end{proof}

\bibliographystyle{plain}
%\bibliography{ref}

\begin{thebibliography}{10}

\bibitem{asikainen2023new}
L.~Asikainen, J.~M. Huusko, and R.~Korhonen.
\newblock A new proximity function estimate on the quotient of the difference and the derivative of a meromorphic function.
\newblock {\em arXiv preprint arXiv:2306.06729}, 2023.

\bibitem{MR2296397}
W.~Bergweiler and J.~K. Langley.
\newblock Zeros of differences of meromorphic functions.
\newblock {\em Math. Proc. Cambridge Philos. Soc.}, 142(1):133--147, 2007.

\bibitem{MR1831783}
W.~Cherry and Z.~Ye.
\newblock {\em Nevanlinna's theory of value distribution}.
\newblock Springer Monogr. Math. Springer-Verlag, Berlin, 2001.
\newblock The second main theorem and its error terms.

\bibitem{cherry2013nevanlinna}
W.~Cherry and Z.~Ye.
\newblock {\em Nevanlinna's theory of value distribution}.
\newblock Springer Science \& Business Media, 2013.

\bibitem{MR2491899}
Y.~M. Chiang and S.~J. Feng.
\newblock On the growth of logarithmic differences, difference quotients and logarithmic derivatives of meromorphic functions.
\newblock {\em Trans. Amer. Math. Soc.}, 361(7):3767--3791, 2009.

\bibitem{MR3695346}
Y.~M. Chiang and X.~D. Luo.
\newblock Difference {N}evanlinna theories with vanishing and infinite periods.
\newblock {\em Michigan Math. J.}, 66(3):451--480, 2017.

\bibitem{MR2220338}
Y.~M. Chiang and S.~N.~M. Ruijsenaars.
\newblock On the {N}evanlinna order of meromorphic solutions to linear analytic difference equations.
\newblock {\em Stud. Appl. Math.}, 116(3):257--287, 2006.

\bibitem{MR220938}
F.~Gross.
\newblock On the distribution of values of meromorphic functions.
\newblock {\em Trans. Amer. Math. Soc.}, 131:199--214, 1968.

\bibitem{MR3206459}
R.~Halburd, R.~Korhonen, and K.~Tohge.
\newblock Holomorphic curves with shift-invariant hyperplane preimages.
\newblock {\em Trans. Amer. Math. Soc.}, 366(8):4267--4298, 2014.

\bibitem{MR2248826}
R.~G. Halburd and R.~J. Korhonen.
\newblock Nevanlinna theory for the difference operator.
\newblock {\em Ann. Acad. Sci. Fenn. Math.}, 31(2):463--478, 2006.

\bibitem{MR0164038}
W.~K. Hayman.
\newblock {\em Meromorphic functions}, volume~78 of {\em Oxford Mathematical Monographs}.
\newblock Clarendon Press, Oxford, 1964.

\bibitem{MR2526825}
R:~Korhonen.
\newblock An extension of {P}icard's theorem for meromorphic functions of small hyper-order.
\newblock {\em J. Math. Anal. Appl.}, 357(1):244--253, 2009.

\end{thebibliography}
\end{document}